\documentclass[11pt]{article}
\usepackage{mathrsfs}
\usepackage{amsfonts,amsmath,dsfont,amssymb,amsthm,stmaryrd}
\usepackage[fit]{truncate}
\usepackage{hyperref}
\hypersetup{colorlinks=true,pdftitle="",pdftex}

\newtheorem{conj}{Conjecture}[section]
\newtheorem{thm}{Theorem}[section]
\newtheorem{rmk}[conj]{Remark}
\newtheorem{lem}[conj]{Lemma}
\newtheorem{prop}[conj]{Proposition}
\newtheorem{coro}[conj]{Corollary}

\newcommand{\mR}{\mathbb{R}} 

\newcommand{\mZ}{\mathbb{Z}}
\newcommand{\mC}{\mathbb{C}}
\newcommand{\mP}{\mathbb{P}}
\newcommand{\mF}{\mathbb{F}}
\newcommand{\mE}{\mathbb{E}}
\def\calG{{\mathcal G}}

\makeatletter
\newcommand{\opnorm}{\@ifstar\@opnorms\@opnorm}
\newcommand{\@opnorms}[1]{%
  \left|\mkern-1.5mu\left|\mkern-1.5mu\left|
   #1
  \right|\mkern-1.5mu\right|\mkern-1.5mu\right|
}
\newcommand{\@opnorm}[2][]{%
  \mathopen{#1|\mkern-1.5mu#1|\mkern-1.5mu#1|}
  #2
  \mathclose{#1|\mkern-1.5mu#1|\mkern-1.5mu#1|}
}
\makeatother

\begin{document}
\title{A combinatorial approach to small ball inequalities for sums and differences}
\author{Jiange Li, Mokshay Madiman\thanks{Both authors are with the Department of Mathematical Sciences, University of Delaware.
This work was supported in part by the U.S. National Science Foundation through grants CCF-1346564 and DMS-1409504 (CAREER).
Email: {\tt lijiange@udel.edu, madiman@udel.edu}}}
\date{\today}
\maketitle


\begin{abstract}
Small ball inequalities have been extensively studied in the setting of Gaussian processes
and associated Banach or Hilbert spaces. In this paper, we focus on studying small ball probabilities 
for sums or differences of independent, identically distributed random elements taking
values in very general sets. Depending on the setting-- abelian or nonabelian groups,
or vector spaces, or Banach spaces-- we provide a collection of inequalities relating
different small ball probabilities that are sharp in many cases of interest. We prove these 
distribution-free probabilistic inequalities
by showing that underlying them are inequalities of extremal combinatorial nature,
related among other things to classical packing problems such as the kissing number problem.
Applications are given to moment inequalities.
\end{abstract}


\section{Introduction}

Given the ubiquity of sums of independent, identically distributed random variables in probability theory (as well as, indirectly, in many other parts of mathematics), it is natural to look for ways to estimate the
probability that a sum lies in a given measurable set. In general, this can be a rather
complex calculation, and is often intractable. The {\it raison d'etre} of this paper is the fact that
it is often much easier to estimate the probability that a symmetric random variable lies in a
symmetric set; so if we can find a way to relate the desired probability to a probability of this type,
then in many circumstances our task is significantly simplified.

The most general setting in which we can talk about sums (and symmetry) is that of 
group-valued random variables, where the group operation represents summation.
Thus, to state our problem more precisely, consider independent and identically distributed random variables
$X$ and $Y$ taking values in a (possibly nonabelian) topological group with group operation $+$ and its Borel $\sigma$-field;
then our problem is to find good bounds on $\mP(X+Y \in F)$ for arbitrary measurable $F$
in terms of $\mP(X-Y\in K)$ for symmetric (i.e., closed with respect to inversion in the group) 
measurable sets $K$. Since the distribution of $X-Y$ is always symmetric
in that it is invariant with respect to inversion of the random variable, this would provide a reduction of the
form mentioned earlier. We also study a related problem, namely that of estimating 
 $\mP(X-Y \in F)$ for arbitrary measurable $F$
in terms of $\mP(X-Y\in K)$ for symmetric measurable sets $K$.

It might seem that the problem stated is somewhat abstruse; however, it is closely related to a number of
influential streams of recent research. To highlight these connections, we discuss the problem from various   
perspectives.

\noindent{\bf Symmetrization.} 
Symmetrization is a basic and powerful meta-technique that arises in many different guises in
different parts of mathematics.  Instances include Steiner symmetrization in convex geometry and the study of isoperimetric phenomena \cite{BZ88:book, Bur09:tut},
Rademacher symmetrization in probability in Banach spaces and empirical process theory \cite{MZ39, JM75, GZ84, Pol89}, 
the use of rearrangements in functional analysis \cite{LL01:book, AFP00:book}, the study
of partial differential equations \cite{Kaw85:book} and probability theory \cite{Wat83, BS01, WM14}, and others too numerous to mention.
One goal of this paper is to develop a symmetrization technique for estimating small ball probabilities of
sums and differences of independent and identically distributed (henceforth, i.i.d.) random vairables.  
We call these small ball probabilities even though there may be no ``ball'' under consideration (for instance,
no norm in the general group settings that we will consider), because when considered in the context 
of finite-dimensional vector spaces, these are related
to inequalities for the probability of lying in a ball with respect to some norm.

\noindent{\bf Concentration functions.} 
The notion of the concentration function was introduced by P. L\'evy, as a means of describing
in a flexible way the spread or concentration of a real-valued random variable that may not have
finite moments. For a real-valued random variable $X$ with distribution $P_X$, the L\'evy
concentration function is given by $Q(X,s)=\sup_{x\in\mR} P_X([x,x+s]))$ for $s>0$.
While there was already much attention paid to concentration functions in classical 
probability theory (see, e.g., \cite{Lev37:book, Doe39, Kol58, Rog61, Ess66, Ess68, Kes69, HT73:book, Kan76, MR80}), their study received 
renewed attention in recent years  \cite{DFY01, DS05, RV08, FS07, EZ13, NV13} 
because of the relevance of arithmetic structure to the concentration
function of linear combinations of i.i.d. random variables,
as well as applications to random matrix theory. While we do not directly 
address the literature on concentration functions in our note, our results
may be seen as providing bounds on multidimensional or non-Euclidean analogs 
of concentration functions in general spaces. Indeed, a natural way to define 
the concentration function in a general setting, say an abelian group $G$,
would be to set
$$
Q(X,F)=\sup_{x\in G} \mP(X\in x+F),
$$
where the set-valued parameter $F$ plays the role of the parameter $s$ in the definition
$Q(X,s)$ of the concentration function for real-valued random variables. Since the constants 
that appear in our results are packing numbers $N(F,K)$ that are invariant with respect to
translations of $F$, our results directly imply concentration function bounds. For instance,
Theorem~\ref{thm:sum-diff} implies the following statement: {\it If $F$ is an arbitrary measurable subset of an
abelian topological group $G$ and $K$ is a measurable subset of $G$ containing the identity
in its interior, then for any pair $X, Y$ of $G$-valued random variables that are independent and identically distributed,
we have}
\begin{align}\label{eq:conc}
Q(X+Y, F) \leq N(F, K)\cdot Q(X-Y, K) .
\end{align}

\noindent{\bf R\'enyi entropy comparisons.} 
By taking balls of vanishing radius in a metrizable, locally compact abelian group $G$, 
and considering independent random variables $X$ and $Y$ drawn from a distribution on $G$ 
that has a density with respect to the Haar measure, the inequality \eqref{eq:conc}
yields an inequality relating the essential supremum
of the density of $X+Y$ and that of the density of $X-Y$. This may be interpreted as
relating the R\'enyi entropies of order $\infty$ of $X+Y$ and $X-Y$. Inequalities for these
quantities have attracted interest from different communities including information theory and 
convex geometry (see, e.g., \cite{MMX17:1, MMX17:2, MWW17:1, MWW17:2}),
and comparing R\'enyi entropies of $X+Y$ and $X-Y$ in particular is closely related to the study
of more-sums-than-differences sets (see, e.g., \cite{ALM17}) as well as
sum-difference inequalities (see, e.g., \cite{MK10:isit, MK18}) in additive combinatorics.

\noindent{\bf Packing problems/extremal combinatorics.}
In 1995, Alon and Yuster \cite{AY95} showed that for i.i.d. real-valued random variables $X$ and $Y$, 
\begin{align}\label{eq:ay}
\mP(|X-Y|\leq b) < (2\lceil b/a\rceil-1) \cdot \mP(|X-Y| \leq a),
\end{align}
thus answering a question raised by Peres and Margulis. 
They also observed that the optimal constants in such estimates are closely related to 
the kissing number problem, which is a long-standing problem in geometry; 
indeed, the kissing number in $\mR^3$ was a subject of discussion between Isaac Newton 
and David Gregory in 1690's. A similar probabilistic inequality proved by Katona \cite{Kat69} 
is closely related to Tur\'an-type theorems for triangle-free graphs.
It turns out that behind the main results of this paper, which among other things
generalize significantly the inequality \eqref{eq:ay} of \cite{AY95},
are actually statements from extremal combinatorics, which we prove en route to proving
our main results. This strengthens the link between extremal combinatorial phenomena
and probabilistic inequalities, in a much more general setting than that of \cite{AY95}, 
in analogy with similar links developed by Katona in a series of papers (see, e.g., \cite{KS80, Kat85}).

\noindent{\bf Moment inequalities.}
Probability bounds are of course closely related to moment inequalities, and our results
in particular can be used to develop a number of moment inequalities for functions of sums and 
differences of random variables under various assumptions on the distribution and/or the function.
Such inequalities are of intrinsic interest since they serve as tools in a variety of areas.

\noindent{\bf Random walks.}
For $0<a<2b$, the following sharp 
symmetrization inequality for i.i.d. real-valued random variables $X$ and $Y$ is proved in \cite{DLL15}:
\begin{align}\label{eq:dll}
\mP(|X+Y|\leq b)<\lceil2b/a\rceil\cdot \mP(|X-Y|\leq a).
\end{align}
For $a\geq2b$, the estimate still holds with $``\leq"$ in the middle. 
This generalizes the earlier work of Siegmund-Schultze and von Weizs\"acher \cite{SW07}, which considered the special case $a=b$
and used it as a key ingredient in studying  the level crossing probabilities for random walks on the real line. 
The results of this paper contain those of \cite{DLL15}, and although we do not investigate this direction
further here, it is conceivable that our results would be useful in the study of random walks on groups.

Having briefly provided motivation from different points of view, let us say something about our methods, especially
as compared with those of \cite{DLL15}, which focused on random variables taking values 
in a separable Banach space and may be seen as a precursor to this paper.
In exploring small ball inequalities for sums and differences on general groups, the reason we 
study the problem from a combinatorial point of view is two-fold: firstly, 
it seems to be impossible to generalize the analytical technique developed in \cite{DLL15}
to the group setting because it relies essentially on the availability of a dilation
operation on the space, and secondly (and perhaps more importantly), it is reasonable 
to expect certain extremal combinatorial nature underlying these probabilistic inequalities, given that they are independent 
of the probability distributions imposed on our random variables. 

In Section~\ref{sec:comb}, we justify this intuition that distribution-free probabilistic estimates
should be connected to extremal combinatorial phenomena, which essentially follows from the Law of Large Numbers. 
Since our findings, not surprisingly, differ depending on how much structure we assume
of our ambient set, we deploy Sections~\ref{sec:ab}, \ref{sec:nab} and \ref{sec:tvs} respectively to explore what can be said 
for abelian topological groups, nonabelian topological groups, and topological vector spaces. In Section~\ref{sec:tight},
we discuss the tightness of our inequalities. Section~\ref{sec:appln} is devoted to
various applications, including moment inequalities in  normed vector spaces.


\section{A combinatorial perspective on distribution-free inequalities}
\label{sec:comb}
 
In this section, we demonstrate a combinatorial approach, which enables us to prove distribution-free probabilistic inequalities by considering their combinatorial analogs. The idea was originally used by Katona \cite{Kat69} to prove probabilistic inequalities using results from graph theory. 

Let $X$ be a random variable taking values in certain measurable space. Let $F$ and $K$ be two measurable subsets of the $k$-fold product space. Given a sequence $X_1, \cdots, X_m$ of independent copies of $X$, the random variable $T_m(X, F)$ is defined to be
$$
T_m(X, F)=|\{(i_1,\cdots, i_k): i_1\neq\cdots\neq i_k, (X_{i_1},\cdots, X_{i_k})\in F\}|.
$$
Similarly we  can define $T_m(X, K)$. Their combinatorial counterparts can be defined for a deterministic sequence $x_1,\cdots, x_m$. More specifically, we define
$$
T_m(F)=|\{(i_1, \cdots, i_k): i_1\neq\cdots\neq i_k, (x_{i_1}, \cdots, x_{i_k})\in F\}|.
$$
We define $T_m(K)$ in a similar manner. Clearly, $T_m(F)$ and $T_m(K)$ depend on the selection of the deterministic sequence. We will not emphasize this when it is clear from context. 

\begin{prop}\label{prop:prob-dis}
Suppose that there is a function $h_k(m)=o(m^k)$ and an absolute constant $C(F, K)$ such that the inequality 
\begin{align}\label{eq:assumption}
T_m(F)\leq h_k(m)+C(F, K)\cdot T_m(K)
\end{align}
holds for any deterministic sequence $x_1,\cdots, x_m$. Then the following inequality
\begin{align}\label{eq:dis-free}
\mP((X_1, \cdots, X_k)\in F)\leq C(F, K)\cdot\mP((X_1, \cdots, X_k)\in K)
\end{align}
holds for any i.i.d. random variables $X_1, \cdots, X_k$.
\end{prop}

\begin{proof}
Since the inequality \eqref{eq:assumption} holds for any deterministic sequences, we have 
$$
T_m(X, F)\leq h_k(m)+C(F, K)\cdot T_m(X, K).
$$
This particularly implies that
\begin{align}\label{eq:expectation}
\mE (T_m(X, F))\leq h_k(m)+C(F, K)\cdot \mE (T_m(X, K)).
\end{align}
Notice that $T_m(X, F)$ can be written as the summation of Bernoulli random variables with the same distribution
\begin{align}\label{eq:rep-txf}
T_m(X, F)=\sum1_{\{(X_{i_1}, \cdots, X_{i_k})\in F\}},
\end{align}
where the summation is taken over all ordered $k$-tuples $(i_1, \cdots, i_k)$ with distinct coordinates. Therefore we have
\begin{align}\label{eq:exp-txf}
\mE (T_m(X, F))=(m)_k\cdot\mP((X_1, \cdots, X_k)\in F),
\end{align}
where the notation $(m)_k$ stands for the product $m(m-1)\cdots(m-k+1)$. Similarly we have
\begin{align}\label{eq:exp-txk}
\mE (T_m(X, K))=(m)_k\cdot\mP((X_1, \cdots, X_k)\in K).
\end{align}
Combining \eqref{eq:expectation}, \eqref{eq:exp-txf} and \eqref{eq:exp-txk}, we have 
$$
\mP((X_1, \cdots, X_k)\in F)\leq \frac{h_k(m)}{(m)_k}+C(F, K)\cdot\mP((X_1, \cdots, X_k)\in K).
$$
Since $m^k$ and $(m)_k$ have the same order for fixed $k$ and large $m$, and $h_k(m)=o(m^k)$, 
the proposition follows by taking the limit $m\rightarrow\infty$.
\end{proof}

Although the above proof is very simple, let us demonstrate the heuristic behind this combinatorial argument. 
We will see that the assumption \eqref{eq:assumption} is not artificial and it has to be true if the inequality \eqref{eq:dis-free} 
holds for all distributions. For simplicity, we denote by $p=\mP((X_1, \cdots, X_k)\in F)$. Using the representation \eqref{eq:rep-txf}, we have
\begin{align*}
\mE\left(\frac{T_m(X, F)}{(m)_k}-p\right)^2 &= \frac{1}{(m)_k^2}\mE\left(\sum(1_{\{(X_{i_1}, \cdots, X_{i_k})\in F\}}-p)\right)^2\\
&= \frac{1}{(m)_k^2}\sum\left(\mE1_{\{(X_{i_1}, \cdots, X_{i_k})\in F\}}\cdot1_{\{(X_{i_1'}, \cdots, X_{i_k'}\in F\}}-p^2\right).
\end{align*}
If the $k$-tuples $(i_1, \cdots, i_k)$ and $(i_1', \cdots, i_k')$ share no common index, we have
$$
\mE\left(1_{\{(X_{i_1}, \cdots, X_{i_k})\in F\}}\cdot1_{\{(X_{i_1'}, \cdots, X_{i_k'})\in F\}}\right)=p^2.
$$
Since there are $o(m^{2k})$ pairs of ordered $k$-tuples $(i_1, \cdots, i_k)$ and $(i_1', \cdots, i_k')$ with common indices, we have
$$
\mE\left(\frac{T_m(X, F)}{(m)_k}-\mP((X_1, \cdots, X_k)\in F)\right)^2\longrightarrow 0,~\text{as}~m\rightarrow\infty.
$$
In particular, we have
$$
\frac{T_m(X, F)}{(m)_k}\xrightarrow{a.s.} \mP((X_1, \cdots, X_k)\in F),~\text{as}~m\rightarrow\infty.
$$
Similar property holds for $T_m(X, K)$. Combining with the inequality \eqref{eq:dis-free}, we have
$$
T_m(X, F)\leq o(m^k)+C(F, K)\cdot T_m(X, K), ~a.s.
$$
Hence for almost all realizations of $X_1, \cdots, X_m$, i.e., deterministic sequences $x_1, \cdots, x_m$, we will have
$$
T_m(F)\leq o(m^k)+C(F, K)\cdot T_m(K).
$$
One may notice that the sequences violating the above inequality depend on the distribution of $X$ and the $o(m^k)$ term may depend on the sequences. However, if the inequality \eqref{eq:dis-free} holds for all distributions, it should be reasonable to expect the validity of \eqref{eq:assumption} for all deterministic sequences.

\section{Small ball inequalities in abelian groups}
\label{sec:ab}

Let $G$ denote an abelian topological group equipped with the Borel $\sigma$-algebra (i.e., the $\sigma$-algebra  generated by the open sets). 
Let $X$ and $Y$ be i.i.d. random variables taking values in $G$. A subset of $G$ is said to be {\it symmetric} if it contains the 
group inverse of each element of this set. 


For two subsets $F, K\subseteq G$, their Minkowski sum $F+K$ is defined as
$$
F+K=\{x+y: x\in F, y\in K\}.
$$
We define the difference set $F-K$ in a similar manner. The generalized entropy number $N(F, K)$ is defined to be the maximal number of elements we can select from $F$ such that the difference of any two distinct elements does not belong to $K$. More precisely, we have
\begin{align}\label{eq:def-nfk}
N(F, K)=\sup\{|S|: S\subseteq F, (S-S)\cap K\subseteq\{0\}\}.
\end{align}
Let $T=\{x_1, x_2,\cdots, x_m\}$ be a multiset (or sequence) of $G$, i.e., the elements of $T$ are selected from $G$ and are not necessary distinct. For any $s\in\mR$, the quantities $T_+(F, s)$ and $T_-(K, s)$ are defined as
\begin{align}\label{eq:def-tfs}
T_+(F, s)=ms+|\{(i, j): i\neq j,  x_i+x_j\in F\}|
\end{align}
and
\begin{align}\label{eq:def-tks}
T_-(K, s)=ms+|\{(i, j): i\neq j,  x_i-x_j\in K\}|.
\end{align}
The relation between these two quantities is given in the following lemma, which is similar in spirit to Lemma 2.1 and Lemma 3.2 in \cite{AY95}.

\begin{lem}\label{lem:tfs-tks}
Suppose that $K$ is a symmetric set with $0\in K$. For $s\geq 2$ and any multiset $T$, we have
$$
T_+(F, s)\leq N(F, K)\cdot T_-(K, 2s).
$$
\end{lem}

\begin{proof}
If $N(F, K)=\infty$, the above statement is obviously true. So we will assume that $N(F, K)$ is finite and prove the lemma by induction on the cardinality of $T$. When counting the cardinality of a multiset, every element counts even if two elements have the same value. For the base case $|T|=1$, we have $T_+(F, s)=s$ and $T_-(K, 2s)=2s$. Since $N(F, K)\geq1$, it is clear that the lemma has to be true. We assume that the lemma holds for any multiset $T$ with cardinality $|T|\leq m-1$. Let $t$ be some non-negative integer such that
$$
\max_{x\in T}|(x+K)\cap T|=t+1.
$$
Here we denote by $(x+K)\cap T$ the multiset consisting of elements of $T$ which lie in $x+K$. We will use similar notations without further clarification. 
Let $x^*\in T$ be an element that achieves the above maximum, and set $T^*=T\backslash\{x^*\}$, where ``$\setminus$" is the standard set subtraction notation. 
(We only throw $x^*$ away, not other elements with the same value). Since $K$ is a symmetric set containing $0$, we have
\begin{align}\label{eq:est-tks}
T_-(K, 2s)=T^*_-(K, 2s)+2s+2t.
\end{align}
We also have
\begin{align}\label{eq:est-tfs}
T_+(F, s)\leq T^*_+(F, s)+s+2|(-x^*+F)\cap T|.
\end{align}
The definition in \eqref{eq:def-nfk} implies that we can select at most $N(F, K)$ elements from $(-x^*+F)\cap T$, say $\{y_1, y_2,\cdots, y_k\}$ with $k\leq N(F, K)$, such that their mutual differences are not in $K$. Therefore we have
\begin{align}\label{eq:cover}
(-x^*+F)\cap T\subseteq\cup_{i}(y_i+K)\cap T.
\end{align}
Combining with \eqref{eq:est-tfs}, we have
\begin{align}\label{eq:est-tfs-1}
T_+(F, s)\leq T^*_+(F, s)+s+2(t+1)N(F, K).
\end{align}
By the induction assumption, the lemma holds for $T^*$. Combining \eqref{eq:est-tks} and \eqref{eq:est-tfs-1}, it is not hard to check that the lemma holds when 
$$
s\geq\frac{2N(F, K)}{2N(F, K)-1},
$$
which is implied by the assumption $s\geq 2$.
\end{proof}

By combining Proposition \ref{prop:prob-dis} with Lemma~\ref{lem:tfs-tks}, we have the following result.

\begin{thm}\label{thm:sum-diff}
Let $F$ and $K$ be measurable subsets of the abelian topological group $G$. Suppose that $K$ is symmetric and contains the identity of $G$ in its interior. 
For any i.i.d. random variables $X$ and $Y$ taking values in $G$, we have
$$
\mP(X+Y \in F)\leq N(F, K)\cdot\mP(X-Y\in K).
$$
\end{thm}


In the following, we study the comparison of $\mP(X-Y\in F)$ and  $\mP(X-Y\in K)$. The definition of $T_-(F, s)$ is given in \eqref{eq:def-tks} with $K$ replaced by $F$. Similar to Lemma \ref{lem:tfs-tks}, we have the following result.

\begin{lem}\label{lem:tfs-tks-1}
Suppose that $K$ is a symmetric set with $0\in K$. For $s\geq2$ and any multiset $T$, we have
$$
T_-(F, s)\leq (1+N(F\backslash K, K))\cdot T_-(K, 2s).
$$
\end{lem}

\begin{proof}
We only need to make a slight modification of the proof of Lemma \ref{lem:tfs-tks}. Let $x^*$ be selected in the same way as before and we set $T^*=T\backslash\{x^*\}$. It is clear that the equation \eqref{eq:est-tks} still holds. Instead of \eqref{eq:est-tfs}, we have
$$
T_-(F, s)\leq T^*_-(F, s)+s+|((x^*+F)\cap T)\backslash\{x^*\}|+|((x^*-F)\cap T)\backslash\{x^*\}|.
$$
Notice the following set-inclusion relations
$$
(x^*+F)\subseteq(x^*+K)\cup(x^*+F\backslash K)
$$
and
$$
(x^*-F)\subseteq(x^*+K)\cup(x^*-F\backslash K).
$$
We use the symmetry assumption of $K$ in the second inclusion relation. Apply the covering argument \eqref{eq:cover} again to $x^*+F\backslash K$ and $x^*-F\backslash K$. We have
\begin{align}\label{eq:est-tfs-3}
T_-(F, s)\leq T^*_-(F, s)+s+2t+2(t+1)N(F\backslash K, K).
\end{align}
Combining \eqref{eq:est-tks} and \eqref{eq:est-tfs-3}, for $s\geq 2$, we will have
$$
T_-(F, s)\leq (1+N(F\backslash K, K))\cdot T_-(K, 2s).
$$
Hence we complete the proof of the lemma.
\end{proof}

Combining Proposition \ref{prop:prob-dis} with Lemma~\ref{lem:tfs-tks-1}, we have the following result.

\begin{thm}\label{thm:diff-diff}
Let $F$ and $K$ be measurable subsets of the abelian topological group $G$. Suppose that $K$ is symmetric and contains the identity of $G$ in its interior. 
For any i.i.d. random variables $X$ and $Y$ taking values in $G$, we have
$$
\mP(X-Y \in F)\leq (1+(F\backslash K, K))\cdot\mP(X-Y\in K).
$$
\end{thm}

We include a result of Katona \cite{Kat69} as another application of this combinatorial argument, 
which is related to Tur\'an's theorem on triangle-free graphs.  
While we provide a reformulation of Katona's proof for completeness, we do not claim any novelty here. 

\begin{thm}[Katona \cite{Kat69}] \label{thm:katona}
Let $X$ and $Y$ be i.i.d. random variables taking values in a Hilbert space $V$ with the norm $\|\cdot\|$. Then we have 
$$
\left(\mP(\|X\|\geq 1)\right)^2\leq2\mP(\|X+Y\|\geq 1).
$$
\end{thm}

\begin{proof}
Let $F, K\subseteq V\times V$ be the subsets defined by
$$
F=\{(x, y): \|x\|\geq 1, \|y\|\geq 1\}
$$ 
and
$$
K=\{(x, y): \|x+y\|\geq 1\}.
$$
Given a multiset $T=\{x_1,\cdots, x_m\}$ of $V$, we define
$$
T_m(F)=|\{(i, j): i\neq j, (x_i, x_j)\in F\}|
$$
and
$$
T_m(K)=|\{(i, j): i\neq j, (x_i, x_j)\in K\}|.
$$
By Proposition \ref{prop:prob-dis}, the statement holds if we can show that
\begin{align}\label{eq:tmf-tmk}
T_m(F)\leq2(m+T_m(K)).
\end{align}
Suppose that there are $n$ elements of $T$ with norms not less than 1. Then we have 
\begin{align}\label{eq:est-tmf}
T_m(F)=n^2-n.
\end{align}
Let us consider a simple graph $\calG$ on these $n$ elements-- 
note that we think of any two elements as different vertices even if they have the same value. 
Two vertices $x$ and $y$ are adjacent if and only if $\|x+y\|\geq1$. Then we have
\begin{align}\label{eq:est-tmk}
T_m(K)\geq 2e(\calG),
\end{align}
where $e(\calG)$ is the number of edges of $\calG$. For any 3 vertices $x, y, z$, there exists at least a pair, say $x, y$, 
such that the angle between them is no more than $2\pi/3$, which implies that $\|x+y\|\geq1$. 
This fact implies that the complementary graph is triangle free. Using Tur\'an's theorem, we have
\begin{align}\label{eq:est-edge}
e(\calG)\geq {n\choose 2}-\frac{n^2}{4}.
\end{align}
Then the estimate \eqref{eq:tmf-tmk} follows from \eqref{eq:est-tmf}, \eqref{eq:est-tmk} and \eqref{eq:est-edge}.
\end{proof}

\begin{rmk}
Katona \cite{Kat69} proved the result for discrete random variables, and then extended it for continuous random variables employing 
continuous versions of Tur\'an-type theorems. This discretization argument was used in a series of papers \cite{Kat69, Kat77, Kat78, Kat80, Kat81, Kat82, Kat83} 
on the optimal estimate of $\mP(\|X+Y\|\geq c)$ in terms of $\mP(\|X\|\geq 1)$; \cite{Kat85} contains a
comprehensive survey. Similar results were independently obtained by Sidorenko \cite{Sid83}. 
\end{rmk}

We do not know whether Katona's result holds in general Banach spaces. This motivates us to prove the following symmetrization result.

\begin{thm}\label{thm:katona-type}
Let $F$ and $K$ be measurable subsets of an abelian topological group $G$. Suppose that $K$ is symmetric and contains the identity of $G$ in its interior. For any $G$-valued i.i.d. random variables $X$ and $Y$, we have
$$
(\mP(X\in F))^2\leq N(F, K)\cdot\mP(X-Y\in K).
$$
\end{thm}

\begin{proof}
The proof is similar to that of Theorem \ref{thm:sum-diff}. We can assume that $N(F, K)$ is finite. For a multiset $T=\{x_1, \cdots, x_m\}$ of $G$, we define 
$$
T(F)=|\{(i, j): i\neq j, x_i\in F, x_j\in F\}|
$$
and
$$
T_-(K)=(N(F, K))^2+2m+|\{(i, j): i\neq j, x_i-x_j\in K\}|.
$$
Then we need to prove the following combinatorial counterpart of the statement
\begin{align}\label{eq:tf-tk}
T(F)\leq N(F, K)\cdot T_-(K).
\end{align}
We will prove it by induction on $m$. It is clear that the statement holds for $m\leq N(F, K)$. Without of loss of generality, we can assume that $x_i\in F$ for all $x_i\in T$. Let $t, x^*, T^*$ be defined in the same way as in Lemma \ref{lem:tfs-tks}. By the definition of $N(F, K)$ in \eqref{eq:def-nfk}, we can select at most $N(F, K)$ elements of $T$ with mutual differences not contained by $K$. The pigeon-hole principle implies that
$$
t+1\geq \frac{m}{N(F, K)}.
$$
One can check that
\begin{align}\label{eq:est-tf}
T(F)=m^2-m=T^*(F)+2(m-1)
\end{align}
and
\begin{align}\label{eq:est-tk}
T_-(K)=T^*_-(K)+2t+2\geq T^*_-(K)+\frac{2m}{N(F, K)}.
\end{align}
The estimate \eqref{eq:tf-tk} follows from \eqref{eq:est-tf}, \eqref{eq:est-tk} and the induction assumption on $T^*$.
\end{proof}

\begin{coro}\label{cor:sym-kat}
Let $a, b>0$. For any i.i.d. real-valued random variables $X$ and $Y$, we have
$$
(\mP(|X|\leq b))^2\leq\lceil 2b/a\rceil\cdot\mP(|X-Y|\leq a).
$$
Moreover the estimate can not be improved.
\end{coro}
The result follows from Theorem \ref{thm:katona-type} with $F=[-b, b]$ and $K=[-a, a]$. The estimate is sharp for the uniform distribution over $\{-b, -b+(1+\epsilon)a, \cdots, -b+(\lceil 2b/a\rceil-1)(1+\epsilon)a\}$ for some small $\epsilon>0$.

To conclude this section, we ask the question of whether there is an analog of Proposition~\ref{prop:prob-dis} when the 
probabilities in comparison are of different magnitudes. 
For example, we do not know a combinatorial counterpart of the distribution-free inequality
$$
\mP((X_1,\cdots, X_k)\in F)\leq f(\mP(X_1,\cdots, X_l)\in K),
$$
where $f$ is a polynomial.


\section{Small ball inequalities in non-abelian groups}
\label{sec:nab}

In this section, let $G$ denote an arbitrary topological group with the identity $e$. We will show that Theorem \ref{thm:sum-diff} 
and Theorem \ref{thm:diff-diff} still hold for certain measurable sets $F$ and $K$ in this general setting. 
Similar to the sumset $F+K$ in the abelian case, define the product set $F\cdot K$ in this non-abelian setting as
$$
F\cdot K=\{xy: x\in F, y\in K\}.
$$
The entropy number $N(F, K)$ is redefined as
\begin{align}\label{eq:def-nfk-1}
N(F, K)=\sup\{|S|: S\subseteq F,  (S\cdot S^{-1})\cap K\subseteq\{e\}\},
\end{align}
where $S^{-1}$ is the set of all inverses of the elements of $S$. For $s\in\mR$ and a multiset $T=\{x_1,\cdots, x_m\}$, the quantities $T_+(F, s)$ and $T_-(K, s)$ are redefined as
\begin{align}\label{def-tfs-1}
T_+(F, s)=ms+|\{(i, j): i\neq j,  x_ix_j\in F\}|
\end{align}
and
\begin{align}\label{def-tks-1}
T_-(K, s)=ms+|\{(i, j): i\neq j,  x_ix_j^{-1}\in K\}|.
\end{align}
Similar to Lemma \ref{lem:tfs-tks}, we have the following result.

\begin{lem}\label{lem:tfs-tks-2}
Suppose $K$ is a normal subgroup of $G$. For $s\geq 2$ and any multi-set $T$, we have
$$
T_+(F, s)\leq N(F, K)\cdot T_-(K, 2s).
$$
\end{lem}

\begin{proof}
The lemma can be proved with a slight modification of the proof of Lemma \ref{lem:tfs-tks}. In order to see how the assumption of $K$ is used, we write the proof again. Let $t$ be some non-negative integer such that 
$$
\max_{x\in T}|(xK)\cap T|=t+1,
$$
where $xK$ is the set of the products of $x$ and the elements of $K$. Let $x^*$ be an element such that the maximum can achieved and $T^*=T\backslash\{x^*\}$. By the definition of $T_-(K, 2s)$, we have
$$
T_-(K, 2s)=T^*_-(K, 2s)+2s+|((Kx^*)\cap T)\backslash\{x^*\}|+|((x^*K^{-1})\cap T)\backslash\{x^*\}|
$$
Since $K$ is a normal subgroup, the estimate of $T_-(K, 2s)$ in \eqref{eq:est-tks} still holds. Similar to \eqref{eq:est-tfs}, we have
$$
T_+(F, s)\leq T^*_+(F, s)+s+|((x^*)^{-1}F)\cap T|+|(F(x^*)^{-1})\cap T|.
$$
Let $\alpha_1, \alpha_2\in F$ be any two elements, and $u_1=(x^*)^{-1}\alpha_1$, $u_2=(x^*)^{-1}\alpha_2$. Since $K$ is a normal subgroup, we can see that $u_1u_2^{-1}\in K$ if only if $\alpha_1\alpha_2^{-1}\in K$. (The assumption that $K$ is a normal subgroup is important here). By the definition of $N(F, K)$ in \eqref{eq:def-nfk-1}, we can select at most $N(F, K)$ elements from $((x^*)^{-1}F))\cap T$, say $\{y_1,\cdots, y_k\}$, such that $y_iy_j^{-1}\notin K$ for any $y_i\neq y_j$. Then we have the following covering relation
$$
((x^*)^{-1}F)\cap T\subseteq\cup_{i}(y_iK)\cap T,
$$
which implies that
$$
|((x^*)^{-1}F)\cap T|\leq(t+1)N(F, K).
$$
Similarly we have the same estimate for $|(F(x^*)^{-1})\cap T|$. Then the estimate of $T_+(F, s)$ is exactly the same as \eqref{eq:est-tfs-1}. So we proved the lemma.
\end{proof}

Using Proposition \ref{prop:prob-dis} together with Lemma~\ref{lem:tfs-tks-2}, we have the following result.

\begin{thm}\label{thm:prod-rat}
Suppose that $K$ is a normal subgroup of the topological group $G$ and the interior of $K$ contains the identity. 
For any i.i.d. random variables $X$ and $Y$ taking values in $G$, and any measurable subset $F$ of $G$, we have
$$
\mP(XY \in F)\leq N(F, K)\cdot\mP(XY^{-1}\in K).
$$
\end{thm}

We omit the proof of the following lemma, since it is essentially the same as that of Lemma \ref{lem:tfs-tks-1}.

\begin{lem}
Suppose that $K\subseteq G$ is a normal subgroup. For $s\geq2$ and any multiset $T$, we have
$$
T_-(F, s)\leq (1+N(F\backslash K, K))\cdot T_-(K, 2s).
$$
\end{lem}

Combining with Proposition \ref{prop:prob-dis}, we have the following result.

\begin{thm}\label{thm:diff-diff-1}
Suppose that $K$ is a normal subgroup of the topological group $G$ and the interior of $K$ contains the identity. 
For any i.i.d. random variables $X$ and $Y$ taking values in $G$, and any measurable subset $F$ of $G$, we have
$$
\mP(XY^{-1} \in F)\leq (1+N(F\backslash K, K))\cdot\mP(XY^{-1}\in K).
$$
\end{thm}


\section{Small ball inequalities in vector spaces}
\label{sec:tvs}

Let $V$ be a topological vector space over a field $\mF$ with the Borel $\sigma$-algebra generated by all open sets. Let $F, K\subseteq V$ be measurable subsets and let $a, b\in\mF$. 
We consider the comparison of $\mP(aX+bY\in F)$ and $\mP(X-Y\in K)$ for i.i.d. random variables $X$ and $Y$ taking values in $V$.

\begin{thm}\label{thm:linear-combi}
Let $F$ and $K$ be measurable subsets of a topological vector space $V$ over a field $\mF$. 
Suppose that $K$ is symmetric and contains the zero vector in its interior. 
Let $a$ and $b$ be non-zero elements of $\mF$. For any i.i.d. random variables $X$ and $Y$ taking values in $V$, we have
$$ 
\mP(aX+bY\in F)\leq N(a, b, F, K)\cdot\mP(X-Y\in K),
$$
where the constant $N(a, b, F, K)$ is defined as
$$
N(a, b, F, K)=\frac{1}{2}\left(N(a^{-1}F, K)+N(b^{-1}F, K)\right).
$$
\end{thm}

\begin{proof}
The proof is essentially the same as that of Theorem \ref{thm:sum-diff}. Let $T=\{x_1,\cdots, x_m\}$ be a multiset of $V$. For $s\in\mR$, we define
$$
T_+(F, s, a, b)=ms+|\{(i, j): i\neq j,  ax_i+bx_j\in F\}|.
$$
By Proposition \ref{prop:prob-dis}, it suffices to prove the following combinatorial analog 
\begin{align}\label{eq:tfs-ab-tks}
T_+(F, s, a, b)\leq N(a, b, F, K)\cdot T_-(K, 2s)
\end{align}
for $s\geq 2$, where $T_-(K, 2s)$ is defined to be the same as \eqref{eq:def-tks}. We select $x^*$ in the same way as in Lemma \ref{lem:tfs-tks} and set $T^*=T\setminus\{x^*\}$. The estimate of $T_-(K, 2s)$ in \eqref{eq:est-tks} still holds. Similar to \eqref{eq:est-tfs}, we have
$$
T_+(F, s, a, b)\leq T^*_+(F, s, a, b)+s+|(-b^{-1}ax^*+b^{-1}F)\cap T|+|(-a^{-1}bx^*+a^{-1}F)\cap T|.
$$
Applying the covering argument \eqref{eq:cover} to $-b^{-1}ax^*+b^{-1}F$, and using the definition of $N(b^{-1}F, K)$, we will have
$$
|(-b^{-1}ax^*+b^{-1}F)\cap T|\leq(t+1)N(b^{-1}F, K).
$$
Similarly we can see that
$$
|(-a^{-1}bx^*+a^{-1}F)\cap T|\leq(t+1)N(a^{-1}F, K).
$$
Hence we have
\begin{align}\label{eq:est-tfs-ab}
T_+(F, s, a, b)\leq T^*_+(F, s, a, b)+s+2(t+1)N(a, b, F, K).
\end{align}
Then the estimate \eqref{eq:tfs-ab-tks} follows from \eqref{eq:est-tks} and \eqref{eq:est-tfs-ab}. So we complete the proof.
\end{proof}

As an easy consequence of Theorem \ref{thm:linear-combi}, we have the following extension of \eqref{eq:ay} and \eqref{eq:dll}. 

\begin{coro}\label{cor:linear-combi}
Let $a, b, c, d$ be non-zero reals and $c, d>0$. For any real-valued i.i.d. random variables $X$ and $Y$, we have
$$
\mP(|aX+bY|\leq c)\leq \frac{1}{2}\left(\left\lceil \frac{2c}{|a|d}\right\rceil+\left\lceil\frac{2c}{|b|d}\right\rceil\right)\mP(|X-Y|\leq d).
$$
\end{coro}

\begin{proof}
We can take $F=[-c, c]$ and $K=[-d, d]$. Elementary geometric argument will yield 
$$
N(a^{-1}F, K)=\left\lceil \frac{2c}{|a|d}\right\rceil,~~~N(b^{-1}F, K)=\left\lceil \frac{2c}{|b|d}\right\rceil.
$$
Then the statement follows from Theorem \ref{thm:linear-combi}.
\end{proof}

\begin{rmk}
For $a=b$, Theorem \ref{thm:linear-combi} and Corollary \ref{cor:linear-combi} match Theorem \ref{thm:sum-diff} and \eqref{eq:dll}, respectively. For $a=-b$, Theorem \ref{thm:linear-combi} and Corollary \ref{cor:linear-combi} are weaker than Theorem \ref{thm:diff-diff} and \eqref{eq:ay}, respectively. This is due to the subtle difference between the covering arguments used in the proofs. The case when $a=0$ or $b=0$ is covered by Corollary \ref{cor:sym-kat}.
\end{rmk}


\section{Discussion of tightness}
\label{sec:tight}

In this section, we study the near extremal distributions for the probabilistic estimates developed in Section \ref{sec:ab}. The discussion will  mainly 
focus on Theorem \ref{thm:sum-diff} and Theorem \ref{thm:diff-diff} for random variables taking values in the Euclidean space $\mR^d$. 
We will see their close connections with the sphere packing problem in combinatorial geometry.

In general it is hard to compute the ratios of $\mP(X\pm Y)\in F$ to $\mP(X-Y)\in K$. If $X$ and $Y$ are assumed to be uniform over a finite set $T=\{x_1, x_2, \cdots, x_n\}$, then we have
$$
\mP(X\pm Y\in F)=\frac{1}{n}\sum_{i=1}^n|(\mp x_i+F)\cap T|
$$
and
$$
\mP(X-Y\in K)=\frac{1}{n}\sum_{i=1}^n|(x_i+K)\cap T|.
$$
If the set $T$ is $K$-separated, i.e., $x_i-x_j\not\in K$ for $i\neq j$, we will have $|(x_i+K)\cap T|=1$ for all $x_i\in T$. We can even make a further assumption that, except $o(n)$ of them, all the sets $\mp x_i+F$ contain the same number of elements of $T$. This is possible if $T$ is selected to consist of certain lattice points. (So we need $X$ and $Y$ to be in a topological vector space $V$). Under these assumptions, we have 
$$
\frac{\mP(X\pm Y\in F)}{\mP(X-Y\in K)}\rightarrow \max_{x\in T}|(\mp x+F)\cap T|,~~\text{as}~n\rightarrow\infty.
$$
Theorem \ref{thm:sum-diff} will be tight if there exists a $K$-separated lattice $\mathcal{L}$ and a point $x\in V$ (not necessary a lattice point) such that $x+F$ contains $N(F, K)$ points of $\mathcal{L}$. The set $T$ can be taken as the union of a subset of $\mathcal{L}$ and the reflection of this subset after certain translation. Similarly, Theorem \ref{thm:diff-diff} is tight if for every lattice point $x\in\mathcal{L}$ the set $x+(F\backslash K)$ contains $N(F\backslash K, K)$ points of $\mathcal{L}$. In this case, we only need to take $T$ to be certain subset of the lattice $\mathcal{L}$. This idea can be used to produce near optimal distributions for the estimates \eqref{eq:dll} and \eqref{eq:ay}. For the estimate \eqref{eq:dll}, we can take $X$ to be uniform over $\{-(n-1)\delta-r, \cdots, -\delta-r, \delta, 2\delta, \cdots, n\delta\}$, where $r>0$ and $\delta>a$. For any $a, b>0$, we can always select appropriate $r$ and $\delta$ such that the ratio of $\mP(|X+Y|\leq b)$ to $\mP(|X-Y|\leq a)$ approaches $\lceil 2b/a\rceil$ as $n\rightarrow\infty$. This example is essentially the same as the one given in \cite{DLL15}. To see the sharpness of \eqref{eq:ay}, we can take $X$ to be uniform over $\{\delta, 2\delta, \cdots, n\delta\}$ for certain $\delta>a$. This example was given in \cite{AY95}.

In the Euclidean space $\mR^d$, let us take $F$ and $K$ to be closed balls centered at the origin of radius $r$ and $1$, respectively. For simplicity, we let $N_+(r)$ and $N_-(r)$ denote $N(F, K)$ and $N(F\backslash K, K)+1$, respectively. Then $N_+(r)$ represents the maximal number of points in a Euclidean ball of radius $r$ with all mutual distances greater than 1. We have an extra restriction on $N_-(r)$  that one of these points should be at the center of the ball. These are the so-called sphere packing problems. The dual problem of $N_+(r)$ asks for the smallest radius $r_+(n)$ of the ball to contain $n$ points with mutual distances at least 1. 
We define $r_-(n)$ in a similar way with the restriction that one of the points should be at the center of the ball. 

Bateman and Erd\H{o}s \cite{BE51} studied the diameters of the extremal configurations of points on the plane. Their results imply the values of $r_+(n)$ and, by the duality, $N_+(r)$ in $\mR^2$. We have the following list of $N_+(r)$ for $r$ in certain range.
$$
N_+(r)=  
\begin{cases}
1, &\text{ if } 0<r\leq 1/2\\
2, &\text{ if } 1/2<r\leq \sqrt 3/3\\ 
3, &\text{ if }  \sqrt 3/3<r\leq \sqrt2/2\\
4, &\text{ if }  \sqrt2/2<r\leq  \frac{1}{2}\csc(\pi/5)\\
5, &\text{ if }   \frac{1}{2}\csc(\pi/5)<r\leq 1\\
7, &\text{ if }  1<r\leq 1+\epsilon,~ \text{small}~ \epsilon>0.\\
\end{cases}
$$
The extremal configurations given by Bateman and Erd\H{o}s are lattice points. Hence the values of $N_+(r)$ 
in the list are tight. It is easy to see that $r_-(2)=\cdots=r_-(7)=1$ with all the points on the unit circle except one point at the center. Bateman and Erd\H{o}s also gave the values of $r_-(n)$ 
for $n=8, 9, 10, 11$. This yields the following list of $N_-(r)$.
$$
N_-(r)=  
\begin{cases}
7, &\text{ if } 1<r\leq \frac{1}{2}\csc(\pi/7)\\
8, &\text{ if }  \frac{1}{2}\csc(\pi/7)<r\leq \frac{1}{2}\csc(\pi/8)\\
9, &\text{ if }  \frac{1}{2}\csc(\pi/8)<r\leq \frac{1}{2}\csc(\pi/9)\\
10, &\text{ if }  \frac{1}{2}\csc(\pi/9)<r\leq \frac{1}{2}\csc(\pi/10).\\
\end{cases}
$$ 

For the sphere packing problems in $\mR^d$, people are generally interested in the optimal packing density $\rho_d$. 
When $r$ is large, it is not hard to see the following asymptotic estimate
$$
N_+(r)\approx N_-(r)=(1+o(1))\frac{\text{vol}(B(r))}{\text{vol}(B(1/2))}\rho_d=(1+o(1))(2r)^d\rho_d.
$$
In $\mR^2$, it is known that hexagonal lattice packing is optimal among all packings (not necessary 
lattice packings) with packing density $\rho_2=\sqrt3\pi/6\approx 0.9069$. 
The sphere packing problem in three-dimensional Euclidean space has a long history. Kepler conjectured in 1611 that 
no arrangement of equally sized spheres can fill the space with a greater average density than that of the 
face-centered cubic and hexagonal close packing arrangements. The density of these arrangements is 
$\sqrt2\pi/6\approx0.7404$. Gauss proved in 1831 that Kepler's conjecture is true if the spheres have 
to be arranged in a regular lattice. The complete proof of Kepler's conjecture was announced only around
20 years ago by Hales (see, e.g., \cite{Hal05}), and a formal verification using automated proof checking was only completed in 2014. 
In the very recent breakthrough work \cite{Via17}, Viazovska proved that the $E_8$ lattice packing gives the 
optimal packing density in dimension 8, and the density is $\rho_8=\pi^4/384\approx0.025367$. 
Building on Viazovska's work, it is shown in \cite{CKMRV17} that Leech lattice packing is optimal in 24 dimension, 
and the packing density is $\rho_{24}=\pi^{12}/12!\approx0.00193$. 

Another interesting problem related to our study is the kissing number problem. In three dimensions it asks how many billiard balls can be arranged so that they all just touch another billiard ball of the same size. This question was a subject of a famous debate between Isaac Newton and David Gregory in 1690's. Newton believed the answer was 12, while Gregory though that 13 might be possible. Generally we can define the $d$-dimensional kissing number $\tau_d$ as the maximal number of points on the unit sphere with Euclidean distances at least 1. For $1<r<1+\epsilon_d$ with small $\epsilon_d>0$, we have the following equation
\begin{align}\label{eq:n-tau}
N_-(r)=\tau_d+1.
\end{align}
The number $\tau_3=12$ was studied by various researchers in the nineteenth century. The best proof now available is due to Leech \cite{Lee56}. The answers $\tau_8=240$ and $\tau_{24}=196,560$ are given by \cite{OS79} and \cite{Lev79}, respectively. It is somewhat surprising that they are technically easier to establish than $\tau_3$. The correct answer $\tau_4=24$ was obtained much later by Musin \cite{Mus03}. For all these results, the extremal configurations follows from lattice packings. Using the relation \eqref{eq:n-tau}, Theorem \ref{thm:diff-diff} can give explicit tight estimates for $r$ slightly greater than 1 in corresponding dimensions. These are all the known values of the kissing number so far. In high dimensions, $\tau_d$ grows exponentially with unknown base. We refer to the monograph \cite{CS99:book} for more discussions of sphere packing problems and their relations with number theory and coding theory.


\section{Applications}
\label{sec:appln}


\subsection{Concentration functions}

Let $G$ be an abelian topological group and let $F\subseteq G$ be a measurable subset. For a random variable $X$ taking values in $G$, the generalized L\'evy's concentration function $Q(X, F)$ is defined to be
\begin{align}\label{eq:con-def}
Q(X, F)=\sup_{x\in G}\mP(X\in x+F).
\end{align}
The main study of concentration functions is devoted to the sum of independent random variables in Banach spaces 
with $F$ taken to be norm balls (see, e.g., \cite{HT73:book, Kan76, MR80}). In the i.i.d. case, Theorem~\ref{thm:sum-diff} 
provides a symmetrization technique to treat different sets and also general groups where no norm may exist. 

For a random variable $X$, we denote by $\widetilde{X}$ the symmetrized random variable $X-Y$, where $Y$ is an independent copy of $X$.

\begin{thm}\label{thm:con-sumdiff}
Let $G$ be an abelian topological group, and let $K$ be a symmetric measurable subset of $G$ such that the interior of $K$ contains the identity of $G$.  
For any independent (not necessarily identical) random variables $X_1, \cdots, X_n$ taking values in $G$, and any measurable subset $F$ of $G$, we have
\begin{align}\label{eq:lienar-com-concen1}
(Q(X_1+\cdots+X_n, F))^2\leq N(F+F, K)\cdot Q(\widetilde{X}_1+\cdots+\widetilde{X}_n, K).
\end{align}
If $X_1, \cdots, X_n$ also have identical distribution, then we have
\begin{align}\label{eq:lienar-com-concen2}
Q(X_1+\cdots+X_n, F)\leq N(F, K)\cdot Q(\widetilde{X}_1+\cdots+\widetilde{X}_{\lfloor n/2\rfloor}, K).
\end{align}
\end{thm}

\begin{proof}
By the definition of generalized L\'evy's concentration function \eqref{eq:con-def}, for any $0<\epsilon<1$, there exists $x\in G$ such that
\begin{align*}
(Q(X_1+\cdots+X_n, F))^2 &\leq (\mP(X_1+\cdots+X_n\in x+F)+\epsilon)^2\\
&\leq \mP(X_1+\cdots+X_n\in x+F)\cdot\mP(X_1'+\cdots+X_n'\in x+F)+3\epsilon\\
&\leq \mP(X_1+\cdots+X_n+X_1'\cdots+X_n'\in 2x+F+F)+3\epsilon\\
&\leq N(F+F, K)\cdot\mP(X_1-X_1'+\cdots+X_n-X_n', K)+3\epsilon\\
&\leq N(F+F, K)\cdot Q(\widetilde{X}_1+\cdots+\widetilde{X}_n, K)+3\epsilon,
\end{align*}
where $(X_1',\cdots, X_n')$ is an independent copy of $(X_1,\cdots, X_n)$, and the 2nd last inequality follows from Theorem \ref{thm:sum-diff}. Then \eqref{eq:lienar-com-concen1} follows by letting $\epsilon\to0$. To prove \eqref{eq:lienar-com-concen2}, it is not hard to see that for independent random variables $X$ and $Y$ 
$$
Q(X+Y, F)\leq Q(X, F).
$$
Therefore we have
$$
Q(X_1+\cdots+X_n, F) \leq Q(X_1+\cdots+X_{2\lfloor n/2\rfloor}, F).
$$
Similarly we can select $x\in G$ such that
\begin{align*}
Q(X_1+\cdots+X_{2\lfloor n/2\rfloor}, F) &\leq \mP(X_1+\cdots+X_{2\lfloor n/2\rfloor}\in x+F)+\epsilon\\
&\leq N(F, K)\cdot \mP(X_1-X_2+\cdots+X_{2\lfloor n/2\rfloor-1}-X_{2\lfloor n/2\rfloor}\in K)+\epsilon\\
&\leq N(F, K)\cdot Q(\widetilde{X}_1+\cdots+\widetilde{X}_{\lfloor n/2\rfloor}, K)+\epsilon.
\end{align*}
In the 2nd inequality we use Theorem \ref{thm:sum-diff}, and \eqref{eq:lienar-com-concen2} follows by letting $\epsilon\to0$.
\end{proof}

If $F$ and $K$ are taken to be balls of vanishing radius in a metrizable, locally compact abelian group, then Theorem \ref{thm:con-sumdiff} 
yields an inequality relating the essential supremums
of the densities of $X+Y$ and $X-Y$. This may be interpreted as relating the R\'enyi entropies of order $\infty$ of $X+Y$ and $X-Y$. 

Let $X$ be a random variable taking values in a locally compact abelian group $G$. Suppose it has density $f$ with respect to the invariant Haar measure $\mu$. For $p\in(0, 1)\cup (1, \infty)$, the R\'enyi entropy of order $p$ (i.e., $p$-R\'enyi entropy) is defined as
$$
h_p(X)=\frac{1}{1-p}\log\int_{G}f(x)^p\mu(dx).
$$
Defining by continuity, $h_1(X)$ corresponds to the classical Shannon entropy. By taking limits, we have
$$
h_0(X)=\log\mu(\mathrm{supp}(f))
$$
and
$$
h_\infty(X)=-\log\|f\|_\infty,
$$
where $\mathrm{supp}(f)$ is the support of $f$, and $\|f\|_\infty$ is the essential supremum of $f$.

We need the following definition before introducing our statement. Let $f$ be a real or complex valued locally integrable function over a metric space with certain reference measure $\mu$. We say that $x$ is a Lebesgue point of $f$ if 
$$
\lim_{r\to0}\frac{1}{\mu(B(x, r))}\int_{B(x, r)}f(y)\mu(dy)=f(x),
$$
where $B(x, r)$ is the ball of radius $r$ centered at $x$.  

\begin{thm}
Let $G$ be a metrizable, locally compact abelian group with the Haar measure $\mu$. Suppose that the following two conditions hold:
\begin{enumerate}
\item For any locally integrable function, almost every point of $G$ is a Lebesgue point.
\item We have $\limsup_{r\to0}N(B_{r}, B_{r})<\infty$, where $B_r$ is the ball of radius $r$ centered at the identity of $G$.
\end{enumerate}
For any i.i.d. absolute continuous random variables $X$ and $Y$ in $G$, we have
$$
h_\infty(X+Y)- h_\infty(X-Y)\geq-\log\left(\limsup_{r\to0}N(B_{r}, B_{r})\right).
$$
\end{thm}
 \begin{proof}
 It is clear that $X+Y$ and $X-Y$ are also absolute continuous random variables. We denote by $f_{X+Y}$ and $f_{X-Y}$ the densities of $X+Y$ and $X-Y$, respectively, with respect to the Haar measure $\mu$. We denote by $\|f_{X+Y}\|_\infty$ and $\|f_{X-Y}\|_\infty$ the essential supremums of $f_{X+Y}$ and $f_{X-Y}$, respectively. We consider the following two possible cases.\\
{\bf Case 1:} $\|f_{X+Y}\|_\infty<\infty$. For any $\epsilon>0$, we can find a Lebesgue point of $f_{X+Y}$, say $z\in G$, such that 
 $$
 f_{X+Y}(z)\geq \|f_{X+Y}\|_\infty-\epsilon.
 $$
The existence of such a point is due to the first assumption. Since $z$ is a Lebesgue point, when $r>0$ is small, we have
 \begin{align*}
 \int_{B(z, r)} f_{X+Y}(w)\mu(dw) &\geq ( f_{X+Y}(z)-\epsilon)\mu(B(z, r))\\
 &\geq (\|f_{X+Y}\|_\infty-2\epsilon)\mu(B_r).
 \end{align*}
In the 2nd inequality we used the invariant property of Haar measure: $\mu(B(z, r))=\mu(B_r)$. Notice that
 \begin{align*}
 \int_{B(z, r)} f_{X+Y}(w)d\mu(w) &\leq Q(X+Y, B_{r})\\
 &\leq N(B_{r}, B_{r}) Q(X-Y, B_{r})\\
 &\leq N(B_{r}, B_{r})\|f_{X-Y}\|_\infty\mu(B_r).
 \end{align*}
 The 2nd inequality follows from Theorem \ref{thm:con-sumdiff}. So we have
 $$
 \|f_{X+Y}\|_\infty-2\epsilon\leq N(B_{r}, B_{r})\|f_{X-Y}\|_\infty.
 $$
As $\epsilon\to 0$, we will have
$$
 \|f_{X+Y}\|_\infty\leq \limsup_{r\to0}N(B_{r}, B_{r})\cdot\|f_{X-Y}\|_\infty,
 $$
 which is equivalent to the desired result. \\
 {\bf Case 2:} $\|f_{X+Y}\|_\infty=\infty$. Similarly for any $M>0$, we can find a Lebesgue point of $f_{X+Y}$, say $z\in G$, such that 
 $$
 f_{X+Y}(z)\geq M.
 $$
For any $\epsilon>0$, the same argument implies that for small $r>0$ we have
$$
M-\epsilon\leq N(B_{r}, B_{r})\|f_{X-Y}\|_\infty.
$$
It is clear that $N(B_{r}, B_{r})\geq 1$. Since $M$ can be arbitrarily large, we have
$$
\|f_{X-Y}\|_\infty=\infty.
$$
In this case the statement trivially holds.
\end{proof}

\begin{rmk}
For the Euclidean space $\mR^d$, the first assumption follows from Lebesgue's differentiation theorem and the second assumption trivially holds since $N(B_{r}, B_{r})$ is independent of the radius $r$. For $d=1$ we know that $N(B_{r}, B_{r})=2$. Therefore for real-valued i.i.d. random variables $X$ and $Y$ we have
$$
h_\infty(X+Y)-h_\infty(X-Y)\geq-\log2.
$$
This is somewhat surprising, since such a result does not hold for R\'enyi entropies of order 0 and 1. 
Indeed, the comparison of $h_p(X+Y)$ and $h_p(X-Y)$ for $p=0, 1$ is closely related to the study of 
more-sums-than-differences sets in additive combinatorics. This is discussed in \cite{ALM17},
where it is also shown that $h_1(X+Y)-h_1(X-Y)$ can be made to take any real value if $X$ and $Y$
are identically distributed on the real line. 
\end{rmk}

If $G$ is a topological vector space, we can apply Theorem \ref{thm:linear-combi} to the study of concentration functions of linear combinations of i.i.d. random variables.

\begin{coro}
Let $F$ and $K$ be measurable subsets of a topological vector space $V$ over a field $\mF$. 
Suppose that $K$ is symmetric and contains the zero vector in its interior. 
Let $a$ and $b$ be non-zero elements of $\mF$. For any i.i.d. random variables $X$ and $Y$ taking values in $V$, we have
$$
Q(aX+bY, F)\leq N(a, b, F, K)\cdot Q(X-Y, K).
$$
\end{coro}

\begin{proof}
By the definition of generalized L\'evy's concentration function \eqref{eq:con-def}, for any $\epsilon>0$, there exists $x\in V$ such that
\begin{align*}
Q(aX+bY, F) &< \mP(aX+bY\in x+ F)+\epsilon\\
&\leq N(a, b, F, K)\cdot \mP(X-Y\in K)+\epsilon\\
&\leq N(a, b, F, K)\cdot Q(X-Y, K)+\epsilon.
\end{align*}
In the 2nd inequality we use Theorem \ref{thm:linear-combi}. Since $\epsilon>0$ is arbitrary, the statement follows by letting $\epsilon\rightarrow0$.
\end{proof}

\begin{rmk}
It is natural to consider possible extensions of the result for more than 2 i.i.d. random variables. For instance, it is an interesting question to 
study the variation of \eqref{eq:lienar-com-concen2} with the left hand side replaced by $Q(a_1X_1+\cdots+a_nX_n, F)$, 
when $X_1, \cdots, X_n$ take values in a vector space and $a_1, \cdots, a_n$ are arbitrary coefficients in a field. 
Such questions are related to the Littlewood-Offord phenomenon (see, e.g., \cite{NV13, MWW17:1}). The simplest 
manifestation of the latter is the fact that when $X_i$ are Bernoulli random variables taking values 0 and 1 with equal probability, 
and $a_1, \ldots, a_n$ are positive integers, 
then 
$$
Q(a_1X_1+\cdots+a_nX_n, 1/2)\leq Q(X_1+\cdots+X_n, 1/2)= O(n^{-1/2}).
$$ 
Moreover, when the sequence of coefficients is not allowed to have additive structur, the asymptotic behavior of the concentration function changes;
for example, if the coefficients are forced to all be distinct, the maximal probability decays like $O(n^{-3/2})$.
While there is a long history of work in this direction as surveyed for example in \cite{NV13, GEZ16}, 
it would be very interesting to understand these questions better in general settings, particularly by identifying
extremal coefficients, which are only known in very few cases.
\end{rmk}


\subsection{Moment inequalities}

Unless otherwise specified, we let $V$ denote a topological vector space over the field $\mF$, which may either be
the real field $\mR$ or the complex field $\mC$. Let $\|\cdot\|$ and $\opnorm{\cdot}$ be two equivalent norms on $V$. 
We denote by $I$ the identity operator from $\left(V, \|\cdot\|\right)$ to $\left(V, \opnorm{\cdot}\right)$. By convention its norm $\opnorm{I}$ is defined as
$$
\opnorm{I}=\sup_{\|x\|=1}\opnorm{x}.
$$
Let $B_1(r)$, $B_2(r)$ be the closed balls centered at the origin of radius $r$ under the gauges $\|\cdot\|$, $\opnorm{\cdot}$, respectively. 
It is clear that $\opnorm{I}$ has the following geometric interpretation
\begin{align}\label{eq:def-I}
\opnorm{I}=\inf\{r>0: B_1(1)\subseteq B_2(r)\}.
\end{align}

\begin{thm}\label{thm:holder}
Let $a, b\in\mF$ be non-zero numbers, and let $q\geq p$ be real numbers such that $pq>0$. 
For any i.i.d. random variables $X$ and $Y$ taking values in $V$, we have
\begin{align}\label{eq:holder}
\left(\mE\opnorm{X-Y}^p\right)^{1/p}\leq 2\opnorm{I}\max\left\{|a|^{-1}, |b|^{-1}\right\}\cdot\left(\mE\|aX+bY\|^q\right)^{1/q}.
\end{align}
\end{thm}

\begin{proof}
We assume that the right hand side of \eqref{eq:holder} is finite. Otherwise the theorem yields a trivial result. 
For $0<p<q$ and $p<q<0$, (equivalently $p<q$ and $pq>0$), the following H\"{o}lder's inequality holds 
$$
\left(\mE\|aX+bY\|^p\right)^{1/p}\leq\left(\mE\|aX+bY\|^q\right)^{1/q}.
$$
Therefore the theorem follows from the case $q=p$. For $p>0$ we have
\begin{align*}
\mE\|aX+bY\|^p &= p\int_0^\infty t^{p-1}\mP(\|aX+bY\|>t)dt\\
&= p\int_0^\infty t^{p-1}\left(1-\mP(\|aX+bY\|\leq t)\right)dt.
\end{align*}
By the geometric interpretation \eqref{eq:def-I} of $\|I\|$ and Theorem \ref{thm:linear-combi},  we have
\begin{align*}
\mP(\|aX+bY\|\leq t) \leq \mP(\opnorm{aX+bY}\leq t\opnorm{I})\leq \mP(\opnorm{X-Y}\leq Ct),
\end{align*}
where the constant $C=2\opnorm{I}\max\left\{|a|^{-1}, |b|^{-1}\right\}$. Hence we have
\begin{align*}
\mE\|aX+bY\|^p &\geq p\int_0^\infty t^{p-1}\left(1-\mP(\opnorm{X-Y}\leq Ct)\right)dt\\
&= p\int_0^\infty t^{p-1}\mP(\opnorm{X-Y}> Ct)dt\\
&= C^{-p}\cdot\mE\opnorm{X-Y}^p,
\end{align*}
which is equivalent to the desired statement. For $p<0$, we have
\begin{align*}
\mE\|aX+bY\|^{p} &= -p\int_0^\infty t^{-p-1}\mP(\|aX+bY\|\leq t^{-1})dt\\
&\leq -p\int_0^\infty t^{-p-1}\mP(\opnorm{aX+bY}\leq \opnorm{I}t^{-1})dt\\
&\leq -p\int_0^\infty t^{-p-1}\mP(\opnorm{X-Y}\leq Ct^{-1})dt\\
&= C^{-p}\left(-p\int_0^\infty t^{-p-1}\mP(\opnorm{X-Y}\leq t^{-1})dt\right)\\
&= C^{-p}\cdot\mE\opnorm{X-Y}^{p}.
\end{align*}
The first inequality follows from the geometric interpretation of $\opnorm{I}$. We apply Theorem \ref{thm:linear-combi} in the second inequality. 
The statement follows by taking the $1/p$-th root.
\end{proof}

As an immediate consequence of Theorem \ref{thm:holder}, we have the following result. 

\begin{coro}\label{cor:blrs}
Let $(E, \|\cdot\|)$ denote a Banach space. For $\gamma>0$ and any pair of $E$-valued i.i.d. random variables $X$ and $Y$, we have
\begin{align}\label{eq:particular}
\mE\|X-Y\|^\gamma\leq2^\gamma\mE\|X+Y\|^\gamma
\end{align}
and
$$
\mE\|X+Y\|^{-\gamma}\leq2^\gamma\mE\|X-Y\|^{-\gamma}.
$$
\end{coro}

The positive moment case of Corollary \ref{cor:blrs} should be compared with a result of Buja, Logan, Reeds and Shepp \cite{BLRS94},
who showed that if $p$ and $\gamma$ are positive reals with $0< \gamma\leq p$ and $1\leq p\leq 2$, then
\begin{equation}\label{eq:blrs}
\mE \|X-Y\|_{p}^\gamma\leq\mE \|X+Y\|_{p}^\gamma
\end{equation}
holds for any i.i.d. random variables $X$ and $Y$ in $\mR^d$, where $\|\cdot\|_p$ refers to the
usual $\ell_p$-norm on $\mR^d$. The negative moment case of Corollary \ref{cor:blrs} should be compared with the recent result 
for the Euclidean norm of Gao \cite{Gao18}, who showed that for $0<\gamma<d$, 
\begin{align}\label{eq:gao}
\mE\|X+Y\|_2^{-\gamma}\leq\mE\|X-Y\|_2^{-\gamma}
\end{align}
holds for any i.i.d. random variables $X$ and $Y$ in $\mR^d$. Corollary \ref{cor:blrs} has the advantage that it works for any norm
and in infinite dimensions, but the inequalities \eqref{eq:blrs} and \eqref{eq:gao} 
have the advantage that when valid, they work with better constants.

Related results were also proved by Mattner \cite{Mat97}.
For $0<\gamma\leq 2$, Mattner \cite{Mat97} proved that if $T$ is any orthogonal linear transformation on $\mR^d$, then
$$
\mE \|X-Y\|_2^\gamma\leq\mE \|X-TY\|_2^\gamma,
$$
where $\|\cdot\|_2$ is the Euclidean norm. 

The question of the sharpness of Corollary \ref{cor:blrs} is interesting and open. 
For $p>2$, \cite{BLRS94} constructed i.i.d. random variables $X$ and $Y$ in $\mR^d$ such that
$$
\mE \|X-Y\|_{p}>\mE \|X+Y\|_{p},
$$
but the ratio their construction gives is just a little bit greater than 1, and so does not reveal whether the estimate \eqref{eq:particular} is tight when $\gamma=1$.

One can also consider a more general class of functions of sums and differences than norms.
A function $\varphi: V\rightarrow\mR$ is said to be {\it quasiconcave} if the super-level set 
$$\{x\in V: \varphi(x)> t \}$$ 
is convex for all $t\in\mR$. 
When $V=\mR$, this coincides with the notion of a {\it unimodal} function.

\begin{prop}\label{prop:varphi}
Let $a, b\in\mF$ be non-zero numbers and let $\varphi$ be a non-negative, symmetric, quasiconcave function on $V$. 
For any i.i.d. random variables $X$ and $Y$ taking values in $V$, we have
$$
\mE\varphi(aX+bY)\leq \mE\varphi\left(\frac{X-Y}{2\max\{|a|^{-1}, |b|^{-1}\}}\right).
$$
\end{prop}

\begin{proof}
Since $\varphi$ is non-negative,  we have
\begin{align*}
\mE\varphi(aX+bY)&=\int_0^\infty\mP\left(\varphi(aX+bY)>t\right)dt \\
&=\int_0^\infty \mP(aX+bY\in\varphi^{-1}(t, \infty))dt,
\end{align*}
where we denote by $\varphi^{-1}(t, \infty)$ the set $\{x\in V: \varphi(x)>t\}$. Since $\varphi$ is symmetric and quasiconcave, 
the set $\varphi^{-1}(t, \infty)$ is symmetric and convex.  Then we can apply Theorem \ref{thm:linear-combi} to have
$$
\mP(aX+bY\in\varphi^{-1}(t, \infty))\leq\mP\left(X-Y\in2\max\{|a|^{-1}, |b|^{-1}\}\varphi^{-1}(t, \infty)\right).
$$
This implies that
\begin{align*}
\mE\varphi(aX+bY) &\leq \int_0^\infty\mP\left(\frac{X-Y}{2\max\{|a|^{-1}, |b|^{-1}\}}\in\varphi^{-1}(t, \infty)\right)dt\\
&=\mE\varphi\left(\frac{X-Y}{2\max\{|a|^{-1}, |b|^{-1}\}}\right).
\end{align*}
\end{proof}

\begin{rmk} 
Taking $\varphi(x)=\|x\|^{p}$ for $p<0$ in Proposition \ref{prop:varphi}, we can recover the negative moment case of Theorem \ref{thm:holder} when the 
two norms are identical. It is natural (given the use of different norms in Theorem \ref{thm:holder}) to consider the comparison of $\mE\varphi(aX+bY)$ and $\mE\phi(X-Y)$ for 
distinct non-negative, symmetric, quasiconcave functions $\varphi$ and $\phi$. 
When $\varphi$ and $\phi$ are homogeneous, one can still apply Theorem \ref{thm:linear-combi} to prove results analogous to Proposition \ref{prop:varphi}. 
Without homogeneity we need information on the behavior of $N(a, b, \varphi^{-1}(t, \infty), \phi^{-1}(t, \infty))$, 
which is a function of $t>0$ rather than a constant; we therefore do not bother to consider this question further. 
\end{rmk}

We now give a moment inequality in the spirit of Theorem \ref{thm:katona-type}, where the
comparison is to the original distribution.

\begin{thm}\label{thm:2-1}
Let $p>0$. For any i.i.d. random variables $X$ and $Y$ taking values in $V$, we have
\begin{align}\label{eq:k-h}
\left(\mE\opnorm{X-Y}^p\right)^{1/p}\leq 2^{1+1/p}\opnorm{I}\cdot\left(\mE\|X\|^p\right)^{1/p}.
\end{align}
\end{thm}

\begin{proof}
We can assume that the right hand side of \eqref{eq:k-h} is finite. 
\begin{align*}
\mE\|X\|^p &= p\int_0^\infty t^{p-1}\mP(\|X\|>t)dt\\
&= p\int_0^\infty t^{p-1}\left(1-\mP(\|X\|\leq t)\right)dt.
\end{align*}
Using the geometric interpretation \eqref{eq:def-I} of $\opnorm{I}$ and Theorem \ref{thm:katona-type},  we have
$$
\mP(\|X\|\leq t) \leq \mP(\opnorm{X}\leq \opnorm{I}t) \leq (\mP(\opnorm{X-Y}\leq 2\opnorm{I} t))^{1/2}.
$$
Hence we have
\begin{align*}
\mE\|X\|^p &\geq p\int_0^\infty t^{p-1}\left(1-(\mP(\opnorm{X-Y}\leq 2\opnorm{I}t))^{1/2}\right)dt\\
&\geq \frac{p}{2}\int_0^\infty t^{p-1}\left(1-\mP(\opnorm{X-Y}\leq 2\opnorm{I} t)\right)dt\\
&= \frac{p}{2}\int_0^\infty t^{p-1}\mP(\opnorm{X-Y}> 2\opnorm{I}t)dt\\
&= 2^{-p-1}\opnorm{I}^{-p}\cdot\mE\opnorm{X-Y}^p.
\end{align*}
The statement follows by taking the $1/p$-th root. In the 2nd inequality, we use the simple fact that $1-t\leq 2(1-\sqrt t)$ for $0\leq t\leq 1$. 
\end{proof}

Combining with a result of Gu\'edon \cite{Gue99}, Theorems \ref{thm:holder} and \ref{thm:2-1} imply a reverse H\"{o}lder type inequality for log-concave random variables. 
A random variable $X$ taking values in a Banach space $(E, \|\cdot\|)$ is called {\it log-concave} if for any non-empty Borel sets $A, B\subseteq E$ and $0<\lambda<1$ we have
 $$
\mP(X\in \lambda A+(1-\lambda)B)\geq\mP(X\in A)^{\lambda}\mP(X\in B)^{1-\lambda} ,
$$
where $\lambda A+(1-\lambda)B$ stands for the Minkowski sum $\{\lambda a+(1-\lambda)b: a\in A, b\in B\}$.
Log-concave distributions are a large class, and include all Gaussian distributions, 
exponential distributions, and the uniform distribution over any compact convex set. 

A reverse H\"{o}lder inequality asserts the equivalence of higher and lower moments of random variables. 
There are many different varieties of reverse H\"older inequalities, such as Khinchine inequalities
or inequalities relating different $L^p$ norms of functions; the survey \cite{BFLM18} contains a discussion of
several of these classes. In particular, reverse H\"older inequalities have found
considerable use in recent years at the interface of probability theory and convex geometry
(see, e.g., \cite{GM11, BM11:it, BM11:aop, BM12:jfa, FMW16, FLM15}),
For our purposes, we focus on situations where there exists a constant $C(p, q)$ depending only on $q\geq p$ such that
$$
(\mE\|X\|^q)^{1/q}\leq C(p, q) (\mE\|X\|^p)^{1/p}
$$
holds for random variables $X$ in a normed measurable space. In general such an inequality does not hold,
but it does hold for random variables with log-concave distributions. 
For example, Borell \cite{Bor74} showed the equivalence between the $p$-th and $q$-th moments of log-concave random variables 
for $q\geq p\geq 1$, while Lata\l a \cite{Lat96} demonstrated that the constant $C(p, q)$ can be independent of $p$. 
Gu\'edon \cite{Gue99} established a reverse H\"{o}lder inequality involving negative moments, which we now describe.
Defining $M_0:=\lim_{p\to0}(\mE\|X\|^p)^{1/p}=\exp(\mE\log\|X\|)$, \cite{Gue99} showed that there is an absolute constant $C_p$ such that for any $-1<p\leq 0$, one has
$$
(\mE\|X\|^p)^{1/p}\geq C_pM_0 .
$$
In fact, for $-1<p\leq 0$, Gu\'edon \cite{Gue99} proved the explicit form
$$
(\mE\|X\|^p)^{1/p}\geq \frac{1+p}{4e}\mE\|X\|.
$$ 
For $p>0$, a corresponding result is not explicitly stated in the literature, but combining the estimates A (with $q=0$) and B in \cite[Theorem 4]{Gue99},
one can prove
$$
(\mE\|X\|^p)^{1/p}\leq C_pM_0 ,
$$
with
$$
C_p=\min_{0\leq x\leq 1}(1-x)^{-2}\left(x+\Gamma(p+1)(1-x)^{1/2}\left(\frac{2}{-\log(1-x)}\right)^p\right)^{1/p}.
$$

An important fact implied by Pr\'ekopa-Leindler inequality is that the sum and difference of independent log-concave random variables are still log-concave. 
Therefore it is reasonable to expect that one may have reverse H\"{o}lder-type inequalities relating $aX+bY$ and $X-Y$ for 
i.i.d. log-concave random variables $X$ and $Y$.

\begin{coro}\label{coro:reverse-holder}
Let $V$ be a real topological vector space associated with two equivalent norms $\|\cdot\|$ and $\opnorm{\cdot}$.
Let $a, b$ be reals such that not both of them are zero. Let $p, q$ be reals such that $0<p\leq q$ or $-1<p<q<0$.  There exists an absolute constant $C(a, b, p, q)$ such that
$$
(\mE\opnorm{X-Y}^q)^{1/q}\leq C(a, b, p, q)\cdot(\mE\|aX+bY\|^p)^{1/p}
$$
holds for any i.i.d. log-concave random variables $X$ and $Y$ taking values in $V$. 
\end{coro}

\begin{proof}
Firstly we assume that $a, b$ are non-zero. For $-1<p<q<0$, Gu\'edon's result implies the existence of a constant $C_p$ such that
$$
(\mE\opnorm{X-Y}^p)^{1/p}\geq C_p(\mE\opnorm{X-Y}^q)^{1/q}.
$$
Then the statement follows from Theorem \ref{thm:holder}. The case $0<p<q$ can be proved in the same manner. 
When one of $a, b$ is zero, the proof proceeds in the same way with the application of Theorem \ref{thm:2-1} instead.
\end{proof}

\begin{rmk} 
We point out that H\"{o}lder or reverse H\"{o}lder inequalities of the form,
$$
(\mE\|X+Y\|^q)^{1/q}\leq c\cdot(\mE\opnorm{X-Y}^p)^{1/p} ,
$$
do not hold, as can be seen by taking real-valued random variables $X$ and $Y$ to be independent and uniform on the interval $[n, n+1]$
with $n$ sufficiently large.
\end{rmk}

\subsection{Positive-definite functions}

This section is focused on the estimate of $\mE\varphi(aX+bY)$, where $\varphi$ is a positive-definite function. The study is independent of the small ball inequalities developed in the previous sections.

Let $G$ be an abelian group, and  $\varphi: G\rightarrow \mC$ be a  Hermitian function, i.e.,
a function satisfying $\overline{\varphi(x)}=\varphi(-x)$. The function $\varphi$ is said to be {\it positive-definite} if, for any $x_1,\cdots, x_n\in G$ and $c_1,\cdots, c_n\in\mC$, we have
$$
\sum_{i, j=1}^n\varphi(x_i-x_j)c_i\overline{c_j}\geq0.
$$
Similarly the Hermitian function $\varphi$ is said to be {\it negative-definite} if the reversed inequality holds under the condition $\sum_{i=1}^nc_i=0$. 
For example, for $0<p\leq2$, the function $e^{-\|x\|^p}$ is positive-definite over the Euclidean space $\mR^d$.  

Bochner's well-known theorem asserts that a continuous positive-definite function $\varphi$ on a locally compact abelian group $G$ can be uniquely represented as the Fourier transform of a positive finite Radon measure $\mu$ on the Pontryagin dual group $G^*$, i.e.,
\begin{align}\label{eq:Bochner}
\varphi(x) = \int_{G^*}\xi(x)\mu(d\xi).
\end{align}
The counterpart of Bochner's theorem is the L\'evy-Khinchin representation formula for a continuous negative-definite function $\varphi$ on $\mR^d$, i.e.,
\begin{align}\label{eq:lk-rep}
\varphi(x)=c+i\langle y_0, x\rangle+q(x)+\int_{\mR^d\backslash\{0\}}\left(1-e^{-i\langle x, y\rangle}-\frac{i\langle x, y\rangle}{1+\|y\|^2}\right)\mu(dy)
\end{align}
where $c\in\mR$, $y_0\in\mR^d$, $q(x)$ is some quadratic form on $\mR^d$ and $\mu$ is a L\'evy measure. The close relations between these two types of functions has been well studied. For example, a function $\varphi$ is negative-definite if and only if $e^{-t\varphi}$ is positive-definite for all $t>0$. This observation goes back to Schoenberg. They are also closely related to another important type of functions, the so-called completely monotone functions. We refer to \cite{Ber08, BCR84:book} for more details in this direction.

\begin{thm}\label{thm:pos-def}
Let $G$ be a locally compact abelian group and let $\varphi: G\rightarrow\mC$ be a continuous positive-definite function. For independent random variables $X$ and $Y$ taking values in $G$ and $m, n\in\mZ$, we have
$$
\left|\mE\varphi(mX+nY)\right|^2\leq\mE\varphi(mX-mX')\mE\varphi(nY-nY'),
$$
where $X'$ and $Y'$ are independent copies of $X$ and $Y$, respectively.
\end{thm}

\begin{proof}
Using Bochner's theorem \eqref{eq:Bochner}, we have
$$
|\mE\varphi(mX+nY)| = \left|\mE\int_{G^*}\xi(mX+nY)\mu(d\xi)\right|=\left|\int_{G^*}\mE\xi(mX)\mE \xi(nY)\mu(d\xi)\right|.
$$
The second equation follows from Fubini's theorem and the assumption that $X, Y$ are independent. By Cauchy-Schwartz inequality, we have
\begin{align*}
&\leq \left(\int_{G^*}|\mE\xi(mX)|^2d\mu(\xi)\int_{G^*}|\mE\xi(nY)|^2d\mu(\xi)\right)^{1/2}\\
&= \left(\int_{G^*}\mE\xi(mX) \overline{\mE\xi(mX)}\mu(d\xi) \int_{G^*}\mE\xi(nY) \overline{\mE\xi(nY)}\mu(d\xi)\right)^{1/2}\\
&= \left(\int_{G^*}\mE\xi(mX) \mE\overline{\xi(mX)}\mu(d\xi) \int_{G^*}\mE\xi(nY)\mE\overline{\xi(nY)}\mu(d\xi)\right)^{1/2}\\
&= \left(\int_{G^*}\mE\xi(mX)\mE\xi(-mX)\mu(d\xi) \int_{G^*}\mE\xi(nY)\mE\xi(-nY)\mu(d\xi)\right)^{1/2}\\
&= \left(\int_{G^*}\mE\xi(mX)\mE\xi(-mX')\mu(d\xi) \int_{G^*}\mE\xi(nY)\mE\xi(-nY')\mu(d\xi)\right)^{1/2}\\
&= \left(\mE\int_{G^*}\xi(mX-mX')\mu(d\xi) \mE\int_{G^*}\xi(nY-nY')\mu(d\xi)\right)^{1/2}\\
&= \left(\mE\varphi(mX-mX') \mE\varphi(nY-nY')\right)^{1/2}.
\end{align*}
We denote by $\overline{\mE\xi(mX)}$ the conjugate of $\mE\xi(mX)$. The equation $\overline{\xi(mX)}=\xi(-mX)$ follows from the fact that $\xi\in G^*$.
\end{proof}

Let $V$ be a locally compact topological vector space over a field $\mF$. Repeating the above argument we can prove the following result.

\begin{thm}
Let $\varphi: V\rightarrow\mC$ be a continuous positive-definite function. For independent random variables $X$ and $Y$ taking values in $V$ and $a, b\in\mF$, we have
$$
\left|\mE\varphi(aX+bY)\right|^2\leq\mE\varphi(aX-aX')\mE\varphi(bY-bY'),
$$
where $X'$ and $Y'$ are independent copies of $X$ and $Y$, respectively.
\end{thm}

Taking $m=n=1$ and $m=1, n=0$ in Theorem \ref{thm:pos-def}, we have the following easy consequences.

\begin{coro}
Let $G$ be a locally compact abelian group and let $\varphi: G\rightarrow\mC$ be a continuous positive-definite function. For i.i.d. random variables $X$ and $Y$ taking values in $G$, we have
$$
|\mE\varphi(X+Y)|\leq\mE\varphi(X-Y).
$$
\end{coro}

\begin{coro}
Let $G$ be a locally compact abelian group and let $\varphi: G\rightarrow\mC$ be a continuous positive-definite function such that $\varphi(0)=1$. For any random variable $X$ taking values in $G$, we have
$$
|\mE\varphi(X)|^2\leq \mE\varphi(X-X'),
$$
where $X'$ is an independent copy of $X$.
\end{coro}

\begin{rmk}
Let $\varphi: \mR^d\rightarrow\mR$ be a continuous negative-definite function. It is shown in \cite{LST13} that for any i.i.d. random variables $X$ and $Y$
$$
\mE\varphi(X-Y)\leq\mE\varphi(X+Y).
$$
Moreover, they showed that $\mE\varphi(X+Y)-\mE\varphi(X-Y)$ is the variance of an integrated centered Gaussian process by employing the L\'evy-Khinchin representation \eqref{eq:lk-rep}. 
\end{rmk}

\end{document}